\documentclass{amsart}

\usepackage[all]{xypic}
\usepackage[centertags]{amsmath}
\usepackage{amsfonts}
\usepackage{amscd}
\usepackage{amssymb}
\usepackage{amsthm}
\usepackage{newlfont}
\usepackage{amsxtra}
 \newtheorem{thm}{Theorem}[section]

 \newtheorem{lem}[thm]{Lemma}
 \newtheorem{prop}[thm]{Proposition}
 \theoremstyle{definition}
 \newtheorem{defn}[thm]{Definition}
 \theoremstyle{remark}
 \newtheorem{rem}[thm]{Remark}
 \theoremstyle{remark}
 
 \theoremstyle{definition}
 \newtheorem{notn}[thm]{Notation}
 \numberwithin{equation}{section}

 \newcommand{\Ver}{\mathrm{Ver}}

 \newcommand{\ord}{\mathrm{ord}}

 \newcommand{\GL}{\mathrm{GL}}
 \newcommand{\PGL}{\mathrm{PGL}}

 \newcommand{\Ed}{\mathrm{Ed}}

 \newcommand{\Tr}{\mathrm{Tr}}

 \newcommand{\Odd}{\mathrm{Odd}}

 \renewcommand{\mod}{\mathrm{mod}}
 \newcommand{\Nr}{\mathrm{Nr}}
 \newcommand{\fa}{\mathfrak a}

 \newcommand{\fb}{\mathfrak b}

 \newcommand{\fr}{\mathfrak r}
 
 \newcommand{\fn}{\mathfrak n}

 \newcommand{\cO}{\mathcal{O}}

 \newcommand{\cA}{\mathcal{A}}
 
 \newcommand{\cG}{\mathcal{G}}

 \renewcommand{\cR}{\mathcal{R}}

 \renewcommand{\cH}{\mathcal{H}}
 \newcommand{\cT}{\mathcal{T}}

 \newcommand{\R}{\mathbb{R}}

 \newcommand{\F}{\mathbb{F}}
 \newcommand{\M}{\mathbb{M}}
 \newcommand{\Q}{\mathbb{Q}}
 \newcommand{\Z}{\mathbb{Z}}

 \renewcommand{\P}{\mathbb{P}}

 \newcommand{\To}{\longrightarrow}
 
 \newcommand{\bs}{\setminus}

 \newcommand{\G}{\Gamma}
 
 \newcommand{\La}{\Lambda}
 \newcommand{\la}{\lambda}

 \newcommand{\norm}[1]{\| #1\|}

\begin{document}

\title[Generators of arithmetic groups]{On generators of arithmetic groups over function fields}

\author{Mihran Papikian}

\address{Department of Mathematics, Pennsylvania State University, University Park, PA 16802}

\email{papikian@math.psu.edu}

\thanks{The author was supported in part by NSF grant DMS-0801208 and Humboldt Research Fellowship.}



\begin{abstract} Let $F=\mathbb{F}_q(T)$ be the field of rational functions with
$\mathbb{F}_q$-coefficients, and $A=\mathbb{F}_q[T]$ be the subring
of polynomials. Let $D$ be a division quaternion algebra over $F$
which is split at $1/T$. Given an $A$-order $\La$ in $D$, we find an
explicit finite set generating $\La^\times$.
\end{abstract}


\maketitle



\section{Introduction} \noindent\textbf{Notation.}\hfill

$\F_q$ = the finite field with $q$ elements;

\hspace{0.3in} \textit{throughout the article $q$ is assumed to be
odd}.

$A=\F_q[T]$, $T$ indeterminate.

$F=\F_q(T)$ = the fraction field of $A$.

$|F|$ = the set of places of $F$.

For $x\in |F|$, $F_x$ = the completion of $F$ at $x$.

$\cO_x = \{z\in F_x\ |\ \ord_x(z)\geq 0\}$ = the ring of integers of
$F_x$.

$\F_x$ = the residue field of $\cO_x$; $\deg(x)=[\F_x:\F_q]$.

For $0\neq f\in A$, $\deg(f)$ = the degree of $f$ as a polynomial in
$T$, and $\deg(0)=+\infty$.

For $f/g\in F$, $\deg(f/g)=\deg(f)-\deg(g)$.

$\ord=-\deg$ defines a valuation on $F$; the corresponding place is
denoted by $\infty$.

$K=F_\infty$.

$\cO=\cO_\infty$.

$\pi=T^{-1}$ = uniformizer at infinity.

\vspace{0.1in}

Let $D$ be a quaternion division algebra over $F$ such that
$D\otimes_F K\cong \M_2(K)$. Let $\La$ be an $A$-order in $D$, i.e.,
$\La$, as a subring of $D$ containing $A$, is also an $A$-module
containing $4$ linearly independent generators over $F$. The
subgroup $\G:=\La^\times$ of units of $\La$ consists of elements
$\la\in \La$ with reduced norm $\Nr(\la)\in\F_q^\times$. It is known
that $\G$ is infinite finitely generated group. The problem of
finding explicit sets of generators for $\G$ naturally arises in the
study of these arithmetic groups. In this paper we develop a method
for finding such explicit sets.

Now we explain what we mean by ``explicit'' and state the main
result of the paper. Let $\fr\in A$ be the discriminant of $D$ -
this is the product of the monic generators of the primes in $A$
where $D$ ramifies; see Section \ref{Sec3}. There exists $\fa\in A$
such that $D$ is isomorphic to the algebra over $F$ with basis
$1,i,j, ij$ satisfying the relations
$$
i^2=\fa,\quad j^2=\fr,\quad ij=-ji
$$
(see Lemma \ref{lemHab}). Next, one needs to find an $A$-basis of
$\La$ in terms of $\{1,i,j,ij\}$. Although our method works for
general $A$-orders, to simplify the exposition, we restrict the
consideration to
$$
\La=A\oplus A i\oplus A j\oplus A ij
$$
which has the simplest presentation in terms of the basis
$\{1,i,j,ij\}$. Intrinsically, this is the level-$\fa$ Eichler order
in $D$. In this case,
$$
\G=\left\{a+bi+cj+dij\ \bigg|\ \begin{matrix} a,b,c,d\in A\\
a^2-\fa b^2-\fr c^2+\fa\fr d^2\in \F_q^\times\end{matrix}\right\}.
$$
Now the problem is to find finitely many quadruples $(a_n, b_n, c_n,
d_n)\in A^4$, $1\leq n\leq N$ for some $N$, such that
$\gamma_n=a_n+b_ni+c_nj+d_nij\in \La$ are in $\G$ and generate $\G$.
Finding minimal explicit sets of generators is very difficult. In
fact, for the analogous problem over $\Q$ only in finitely many
cases an explicit minimal set of generators is known, cf. \cite{AB}.
(See \cite{PapLocProp} for some examples in the function field
case.) Instead of trying to find minimal sets of generators for
$\G$, we look for possibly larger sets which are easy to describe.
\begin{defn}
Put $\deg(0)=0$ and for $\la=a+bi+cj+dij\in \La$ define
$$
\norm{\la}=\max(\deg(a), \deg(b), \deg(c),\deg(d)).
$$
\end{defn}

The main result of the paper is the following:

\begin{thm}\label{thm1} The finite subset
$$
S=\{\gamma\in \G\ |\ \norm{\gamma}\leq 4\deg(\fa\fr)+6\}
$$
generates $\G$.
\end{thm}

To prove Theorem \ref{thm1} we use the action of $\G$ on the
Bruhat-Tits tree $\cT$ of $\PGL_2(K)$. The key is to quantify the
discontinuity of the action of $\G$ on $\cT$. More precisely, for a
ball $\cT_B$ of radius $B$ around a fixed vertex and $\gamma\in \G$,
we show that $\gamma \cT_B\cap \cT_B=\emptyset$ once
$\norm{\gamma}>2B$. To deduce Theorem \ref{thm1} from this result,
one needs strong bounds on the diameter of the quotient graph $\G\bs
\cT$. Such a bound follows from the property of $\G\bs \cT$ being
covered by a Ramanujan graph.

A problem of similar nature over $\Q$ was considered by Chalk and
Kelly in \cite{CK}, although the bound they obtain is exponentially
worse than ours. The approach in \cite{CK} is analytic in nature and
relies on the study of isometric circles of $\G$ (see also
\cite{Johansson}).

\begin{rem}
Theorem \ref{thm1} raises the following interesting question: What
is the minimal $\sigma$ independent of $\G$ such that there exists a
constant $\delta$ (also independent of $\G$) with the property that
the set
$$
\{\gamma\in \G\ |\ \norm{\gamma}\leq \sigma \deg(\fa\fr)+\delta\}
$$
generates $\G$? The theorem gives $\sigma \leq 4$. We have a trivial
lower bound $1/4\leq \sigma$, since according to Proposition
\ref{prop2.5} the cardinality of a minimal set of generators for
$\G$ is approximately $q^{\deg(\fa\fr)}$.
\end{rem}

\begin{rem}
After this article was essentially completed, Ralf Butenuth informed
me that he obtained a result similar to Theorem \ref{thm1} by a
different method.
\end{rem}


\section{Arithmetic of quaternion algebras}\label{Sec3} In this section we
recall some facts about quaternion algebras. The standard reference
for this material is \cite{Vigneras}.

Let $D$ be a \textit{quaternion algebra} over $F$, i.e., a
$4$-dimensional $F$-algebra with center $F$ which does not possess
non-trivial two-sided ideals. A quaternion algebra is either a
division algebra or is isomorphic to the algebra of $2\times 2$
matrices. If $L$ is a field containing $F$, then $D\otimes_F L$ is a
quaternion algebra over $L$. Let $x\in |F|$ and denote
$D_x:=D\otimes_F F_x$. We say that $D$ \textit{ramifies} (resp.
\textit{splits}) at $x$ if $D_x$ is a division algebra (resp. is
isomorphic to $\M_2(F_x)$). Let $R\subset |F|$ be the set of places
where $D$ ramifies. It is known that $R$ is a finite set of even
cardinality, and conversely, for any choice of a finite set
$R\subset |F|$ of even cardinality there is a unique, up to an
isomorphism, quaternion algebra ramified exactly at the places in
$R$. In particular, $D\cong \M_2(F)$ if and only if $R=\emptyset$.

Explicitly quaternion algebras can be given as follows.  For $a,
b\in F^\times$, let $H(a,b)$ be the $F$-algebra with basis $1,i,j,
ij$ (as an $F$-vector space), where $i,j$ satisfy
$$
i^2=a,\quad j^2=b,\quad ij=-ji.
$$
$H(a,b)$ is a quaternion algebra, and any quaternion algebra $D$ is
isomorphic to $H(a,b)$ for some $a,b\in F^\times$ (although $a$ and
$b$ are not uniquely determined by $D$, e.g. $H(a,b)\cong H(b,a)$).

From now on we assume that $D$ is a division algebra (equiv. $R\neq
\emptyset$). Let $L\neq F$ be a non-trivial field extension of $F$.
Then $L$ embeds into $D$, i.e., there is an $F$-isomorphism of $L$
onto an $F$-subalgebra of $D$, if and only if $[L:F]=2$ and places
in $R$ do not split in $L$.

There is a canonical involution $\alpha\mapsto \alpha'$ on $D$ which
is the identity on $F$ and satisfies $(\alpha\beta)'=\beta'\alpha'$.
The \textit{reduced trace} of $\alpha$ is
$\Tr(\alpha)=\alpha+\alpha'$; the \textit{reduced norm} of $\alpha$
is $\Nr(\alpha)=\alpha\alpha'$; the \textit{reduced characteristic
polynomial} of $\alpha$ is
$$f(x)=(x-\alpha)(x-\alpha')=x^2-\Tr(\alpha)x+\Nr(\alpha).$$
If $\alpha\not \in F$, then the reduced trace and norm of $\alpha$
are the images of $\alpha$ under the trace and norm of the quadratic
field extension $F(\alpha)/F$.

From now on we also assume that $\infty\not \in R$. For $x\in
|F|-\infty$, denote by $(x)$ the prime ideal of $A$ corresponding to
$x$. Let $f_x\in A$ be the monic generator of $(x)$, and let
$\fr=\prod_{x\in R}f_x$. Given $a, b\in A$, with $b$ irreducible and
coprime to $a$, let
$$
\left(\frac{a}{b}\right)=\left\{
                           \begin{array}{ll}
                             1, & \hbox{if $a$ is a square mod $(b)$} \\
                             -1, & \hbox{otherwise}
                           \end{array}
                         \right.
$$
be the Legendre symbol. Let
$$
\Odd(R)=\left\{
         \begin{array}{ll}
           1, & \hbox{if $\deg(x)$ is odd for all $x\in R$;} \\
           0, & \hbox{otherwise.}
         \end{array}
       \right.
$$

\begin{lem}\label{lemHab}\hfill
\begin{enumerate}
\item Suppose $\Odd(R)=0$. There is a monic irreducible polynomial $\fa\in A$ of even degree
which is coprime to $\fr$ and satisfies
$$
\left(\frac{\fa}{f_x}\right)=-1\text{ for all }x\in R.
$$
For such $\fa$, $D\cong H(\fa, \fr)$.
\item Suppose $\Odd(R)=1$. Let $\xi\in \F_q$
be a non-square. Then $D\cong H(\xi, \fr)$.
\end{enumerate}
\end{lem}
\begin{proof} The proof of this lemma is quite standard; we give the
details for completeness.

(1) By the Chinese Remained Theorem \cite[Prop. 1.4]{Rosen}, there
exists $a\in A$ such that $(a/f_x)=-1$ for all $x\in R$. Consider
the set of polynomials $\{a+\fr b\ |\ b\in A\}$. By \cite[Thm.
4.8]{Rosen}, this set contains irreducible monic polynomials $\fa$
of even degrees. Fix such $\fa$. It is clear that $\fa$ satisfies
$(\fa/f_x)=-1$ for all $x\in R$. Next, by the Reciprocity Law
\cite[Thm. 3.3]{Rosen}
$$
\left(\frac{f_x}{\fa}\right)=(-1)^{\frac{q-1}{2}\deg(\fa)\deg(f_x)}\left(\frac{\fa}{f_x}\right).
$$
Since $\deg(\fa)$ is even, the right hand-side is equal to $-1$.
Hence
\begin{equation}\label{eq-frfa}
\left(\frac{\fr}{\fa}\right)=\prod_{x\in
R}\left(\frac{f_x}{\fa}\right)=(- 1)^{\# R}=1.
\end{equation}

Now we show that $D\cong H(\fa, \fr)$. It is enough to check that
$H(\fa,\fr)$ is ramified exactly at the places in $R$. For this we
need to show that the Hilbert symbol $(\fa,\fr)_x$ is $-1$ if and
only if $x\in R$, cf. \cite[p. 32]{Vigneras}. By \cite[p.
217]{SerreLF}, for $x\in |F|-(\fa)-\infty$, $(\fa,\fr)_x=1$ if and
only if $\fa^{\ord_x(\fr)}$ is a square modulo $(x)$. Now
$\ord_x(\fr)=0$ if $x\in |F|-R-\infty$, and $\ord_x(\fr)=1$ if $x\in
R$. Observe that the image of $\fa$ is not a square in $\F_x$ for
$x\in R$ by the choice of $\fa$. Therefore, $H(\fa, \fr)$ is
ramified at the places in $R$ and is unramified at
$|F|-R-\infty-(\fa)$. By the same argument, $H(\fa,\fr)$ is
unramified at $(\fa)$, since $\fr^{\ord_{(\fa)}(\fa)}=\fr$ is a
square modulo $(\fa)$ by (\ref{eq-frfa}). Finally, $H(\fa, \fr)$ is
unramified at $\infty$ since the number of places where a quaternion
algebra ramifies is even.

(2) Note that $\xi$ in not a square in $\F_x$, $x\in R$, since
$\deg(x)$ is odd. Now apply the argument in the previous paragraph.
\end{proof}

\begin{defn}
Let $\cR$ be a Dedekind domain with quotient field $L$ and let $B$
be a quaternion algebra over $L$. For any finite dimensional
$L$-vector space $V$, an \textit{$\cR$-lattice} in $V$ is a finitely
generated $\cR$-submodule $M$ in $V$ such that $L\otimes_\cR M\cong
V$. An \textit{$\cR$-order} in $B$ is a subring $\La$ of $B$, having
the same unity element as $B$, and such that $\La$ is an
$\cR$-lattice in $B$. A \textit{maximal $\cR$-order} in $B$ is an
$\cR$-order which is not contained in any other $\cR$-order in $B$.
\end{defn}

Let $\La$ be an $A$-order in $D$. It is known that $\La$ is maximal
if and only if $\La_x:=\La\otimes_A \cO_x$ is a maximal
$\cO_x$-order in $D_x$ for all $x\in |F|-\infty$. A maximal
$\cO_x$-order in $D_x$ is unique if $x\in R$ - it is the integral
closure of $\cO_x$ in $D_x$. On the other hand, for $x\not\in R$,
$\La_x$ is maximal if and only if there is an invertible element
$u\in \M_2(F_x)$ such that $u\La_x u^{-1}=\M_2(\cO_x)$.

\begin{defn}
Suppose $\fn\in A$ is square-free and coprime to $\fr$. $\La$ is an
\textit{Eichler order of level-$\fn$} if $\La_x$ is maximal for all
$x\in R$, and for $x\in |F|-R-\infty$ it is isomorphic to the
subring of $\M_2(\cO_x)$ given by the matrices
$$
\left\{ \begin{pmatrix} a & b \\ \fn c & d\end{pmatrix}\ |\
a,b,c,d\in \cO_x\right\}.
$$

Suppose $\La=Ae_1\oplus Ae_2\oplus Ae_3\oplus Ae_4$, where
$e_1,\dots, e_4$ is a basis of $D$ as an $F$-vector space. The
\textit{discriminant} of $\La$ is the ideal of $A$ generated by
$\det(\Tr(e_ie_j))_{i,j}$. It is known that the discriminant of any
order is divisible by $(\fr)^2$. Moreover, the maximal orders are
uniquely characterized by the fact that their discriminant is
$(\fr)^2$, and the level-$\fn$ Eichler orders are uniquely
characterized by the fact that their discriminants are equal to
$(\fn\fr)^2$. By a theorem of Eichler, since $D$ splits at $\infty$,
all maximal $A$-orders are conjugate in $D$; the same is true also
for the level-$\fn$ Eichler orders.
\end{defn}

\begin{defn}\label{defSO}
The order $\La=A\oplus A i\oplus Aj\oplus Aij$ in $H(\fa,\fr)$ will
be called the \textit{standard order}. By computing its
discriminant, we see that $\La$ is a level-$\fa$ Eichler order. In
particular, $\La$ is maximal if and only if $\fa\in \F_q^\times$.
\end{defn}

Given an $A$-order $\La$, the group $\La^\times$ of its invertible
elements consists of
$$
\{\la\in \La\ |\ \Nr(\la)\in \F_q^\times\}.
$$
If $\la\in \La^\times$ is a torsion element, then it is algebraic
over $\F_q$. This easily implies that $\la$ is torsion if and only
if $\Tr(\la)\in \F_q$; such elements will be called
\textit{elliptic}. An element $\la\in \La^\times$ which is not
elliptic will be called \textit{hyperbolic}.

\begin{lem}
If $\la$ is hyperbolic, then its image in $\GL_2(K)$ under an
embedding $D\hookrightarrow \M_2(K)$ has two distinct $K$-rational
eigenvalues.
\end{lem}
\begin{proof}
The reduced characteristic polynomial of $\la$ is
$h_\la:=x^2+\Tr(\la)x+\kappa$, where $\kappa\in \F_q^\times$. Since
$\la\not\in F^\times$, $F(\la)$ is quadratic, and therefore $h_\la$
is irreducible. Next, since $s:=\Tr(\la)\in A$ has degree $\geq 1$,
$s^2-4\kappa$ is a non-zero polynomial of even degree whose leading
coefficient is a square. This implies that $h_\la$ splits over $K$
and has distinct roots.
\end{proof}

Denote
$$
g(R)=1+\frac{1}{q^2-1}\prod_{x\in
R}(q^{\deg(x)}-1)-\frac{q}{q+1}\cdot 2^{\# R-1}\cdot \Odd(R).
$$

Let $\La$ be the standard order in $H(\fa,\fr)$, where $\fa$ is a
monic irreducible polynomial of even degree if $\Odd(R)=0$, and
$\fa=\xi$ is a non-square in $\F_q$ if $\Odd(R)=1$, cf. Lemma
\ref{lemHab}. Denote $\G:=\La^\times$. Note that $\F_q^\times$ is in
the center of $\G$.

\begin{prop}\label{prop2.5}\hfill
\begin{enumerate}
\item $\G/\F_q^\times$ has non-trivial torsion if and only
if $\Odd(R)=1$.
\item Suppose $\Odd(R)=1$. $\G$ can be generated by
$g(R)+2^{\#R-1}$ elements.
\item Suppose $\Odd(R)=0$. $\G/\F_q^\times$ is a free group of rank
$$
1+(q^{\deg(\fa)}+1)(g(R)-1).
$$
\end{enumerate}
\end{prop}
\begin{proof}
(1) and (2) follow from \cite[Thm. 5.7]{PapLocProp}. (3) Let
$\Upsilon$ be a maximal order containing $\La$. Denote
$\G'=\Upsilon^\times$. By \cite[Thm. 5.7]{PapLocProp},
$\G'/\F_q^\times$ is a free group of rank $g(R)$. It is not hard to
show that $[\G'/\F_q^\times:\G/\F_q^\times]=q^{\deg(\fa)}+1$, cf.
\cite[p. 212]{Miyake}. Now the claim that $\G/\F_q^\times$ is a free
group of given rank follows from Schreier's theorem \cite[p.
29]{SerreT}.
\end{proof}


\section{Geometry on the Bruhat-Tits tree}\label{sGBT}

We start by recalling some of the terminology from \cite{SerreT}.
Let $\cG$ be an (oriented) connected graph; see \cite[Def. 1, p.
13]{SerreT}. We denote by $\Ver(\cG)$ and $\Ed(\cG)$ the sets of
vertices and oriented edges of $\cG$, respectively. For $e\in
\Ed(\cG)$, $o(e), t(e)\in \Ver(\cG)$ and $\bar{e}\in \Ed(\cG)$
denote its origin, terminus and inversely oriented edge. We will
assume that for any $v\in \Ver(\cG)$ the number of edges $e$ with
$o(e)=v$ is finite; this number is the \textit{degree} of $v$. $\cG$
is \textit{$m$-regular} if every vertex in $\cG$ has degree $m$. The
\textit{distance} $d(v,w)$ between $v,w\in \Ver(\cG)$ in $\cG$ is
the obvious combinatorial distance, i.e., the number of edges in a
shortest path without backtracking connecting $v$ and $w$. The
\textit{diameter} $D(\cG)$ of a finite graph $\cG$ is the maximum of
the distances between its vertices. A graph in which a path without
backtracking connecting any two vertices $v$ and $w$ is unique is
called a \textit{tree}; the unique path between $v, w$ is called
\textit{geodesic}.

Let $\G$ be a group acting on a graph $\cG$, i.e., $\G$ acts via
automorphisms. We say that $v, w\in \Ver(\cG)$ are
$\G$-\textit{equivalent} if there is $\gamma \in \G$ such that
$\gamma v=w$. $\G$ acts with \textit{inversion} if there is
$\gamma\in \G$ and $e\in \Ed(\cG)$ such that $\gamma e=\bar{e}$. If
$\G$ acts without inversion, then we have a natural quotient graph
$\G\bs\cG$ such that $\Ver(\G\bs \cG)=\G\bs \Ver(\cG)$ and
$\Ed(\G\bs \cG)=\G\bs \Ed(\cG)$, cf. \cite[p. 25]{SerreT}.

Recall the notation $K:=F_\infty$ and $\cO:=\cO_\infty$. Let $V$ be
a $2$-dimensional $K$-vector space. Let $\La$ be an $\cO$-lattice in
$V$. For any $x\in K^\times$, $x\La$ is also a lattice. We call
$\La$ and $x\La$ equivalent lattices. The equivalence class of $\La$
is denoted by $[\La]$.

Let $\cT$ be the graph whose vertices $\Ver(\cT)=\{[\La]\}$ are the
equivalence classes of lattices in $V$, and two vertices $[\La]$ and
$[\La']$ are adjacent if we can choose representatives $L\in [\La]$
and $L'\in [\La']$ such that $L'\subset L$ and $L/L'\cong \F_q$. One
shows that $\cT$ is an infinite $(q+1)$-regular tree. This is the
\textit{Bruhat-Tits tree} of $\PGL_2(K)$.

Fix a vertex $[\La]$ of $\cT$. The set of vertices of $\cT$ at
distance $n$ from $[\La]$ is in natural bijection with
$\P^1(L/\pi^nL)$. An \textit{end} of $\cT$ is an equivalence class
of half-lines, two half-lines being equivalent if they differ in a
finite graph. Taking the projective limit over $n$, we get a
bijection
$$
\partial \cT:=\text{set of ends of }\cT\cong \P^1(\cO)=\P^1(K),
$$
which is independent of the choice of $[\La]$. Given two vectors
$f_1, f_2$ spanning $V$, we denote $[f_1, f_2]=[\cO f_1\oplus \cO
f_2]$. $\GL_2(K)$, as the group of $K$-automorphisms of $V$, acts on
the Bruhat-Tits tree $\cT$ via $g[f_1, f_2]=[gf_1, gf_2]$. It is
easy to check that this action preserves the distance between any
two vertices, and the induced action on $\partial\cT$ agrees with
the usual action of $\GL_2(K)$ on $\P^1(K)$ through fractional
linear transformations.

Fix the standard basis $e_1=\begin{pmatrix} 1\\ 0 \end{pmatrix}$,
$e_2=\begin{pmatrix} 0\\ 1 \end{pmatrix}$ of $V$, and let $O:=[e_1,
e_2]$. Since $\GL_2(K)$ acts transitively on the vertices of $\cT$
and the stabilizer of $O$ is $\GL_2(\cO)K^\times$, we have a
bijection
\begin{align}\label{eq-map}
\GL_2(K)/\GL_2(\cO)K^\times &\overset{\sim}{\To} \Ver(\cT)\\
g &\nonumber \mapsto g\cdot O.
\end{align}
Using the Iwasawa decomposition, one easily sees that the set of
matrices
\begin{equation}\label{eq-setM}
\left\{\begin{pmatrix} \pi^k & u \\ 0 & 1\end{pmatrix}\ \bigg|\
\begin{matrix} k\in \Z\\ u\in K,\ u\ \mod\ \pi^k\cO\end{matrix}\right\}
\end{equation}
is a set of representatives for $\Ver(\cT)$; see \cite[p.
370]{Improper}. The map (\ref{eq-map}) becomes
$$
\begin{pmatrix} \pi^k & u \\ 0 & 1\end{pmatrix}\mapsto [\pi^k e_1,
ue_1+e_2].
$$
Note that under this bijection the identity matrix corresponds to
$O$. We say that a matrix $M\in \GL_2(K)$ is in \textit{reduced
form} if it belongs to the set of matrices in (\ref{eq-setM}). For
two matrices $M,M'\in\GL_2(K)$, we write $M\sim M'$ if they
represent the same vertex in $\cT$.

\begin{lem}\label{lem2.1}
The distance between $\begin{pmatrix} \pi^k & u \\ 0 &
1\end{pmatrix}$ $($in reduced form$)$ and $\begin{pmatrix} 1 & 0 \\
0 & 1\end{pmatrix}$ is
$$
\left\{
  \begin{array}{ll}
    |k|, & \hbox{if $u=0$ or $\ord(u)\geq 0$;} \\
    k-2\cdot \ord(u), & \hbox{if $u\neq 0$ and $\ord(u)<0$.}
  \end{array}
\right.
$$
$($Note that for a matrix in reduced form $k>\ord(u)$.$)$
\end{lem}
\begin{proof}
This is an easy calculation.
\end{proof}

Given two distinct points $P,Q\in \P^1(K)$, there is a unique path
in $\cT$, without backtracking and infinite in both directions,
whose ends are $P$ and $Q$; this is the \textit{geodesic} $\cA(P,Q)$
connecting the two boundary points of $\cT$. For example, the
geodesic connecting $0=(0:1)$ and $\infty=(1:0)$ is the subgraph of
$\cT$ with vertices $\left\{\begin{pmatrix} \pi^k & 0 \\ 0 &
1\end{pmatrix}\ \bigg|\ k\in \Z\right\}$.

Assume $\gamma\in \GL_2(K)$ has two distinct $K$-rational
eigenvalues $a$ and $b$. The eigenvectors corresponding to $a$ and
$b$ can be regarded as two well-defined points on $\P^1(K)$ - if
$\begin{pmatrix} x\\ y
\end{pmatrix}$ is an eigenvector, then the corresponding point is $(x:y)$.
Let $\cA(\gamma)$ be the geodesic connecting these points. Suppose
$\begin{pmatrix} x_1\\ y_1 \end{pmatrix}$ and $\begin{pmatrix} x_2\\
y_2 \end{pmatrix}$ are eigenvectors corresponding to $a$ and $b$,
respectively. Since $\begin{pmatrix} x_1 & x_2 \\ y_1 &
y_2\end{pmatrix}$ maps $\infty$ to $(x_1:y_1)$ and $0$ to
$(x_2:y_2)$,
\begin{equation}\label{eq-geod}
\Ver(\cA(\gamma))=\left\{\begin{pmatrix} x_1 & x_2 \\ y_1 &
y_2\end{pmatrix}\begin{pmatrix} \pi^k & 0 \\ 0 & 1\end{pmatrix}\ |\
k\in \Z\right\}.
\end{equation}
The \textit{distance} from $\cA(\gamma)$ to $O$ is the minimum of
the distances from the vertices on $\cA(\gamma)$ to $O$.

\begin{lem}\label{lem-geod} The action of $\gamma$ on $\cT$ induces translation on
$\cA(\gamma)$ of amplitude $|\ord(a)-\ord(b)|$.
\end{lem}
\begin{proof}
This is easy to see after choosing the eigenvectors of $\gamma$ as a
basis of $V$.
\end{proof}

\begin{lem}\label{lem-BB}
Suppose $x_1y_1\neq 0$ and $x_2y_2\neq 0$. Denote $x:=x_1/y_1$ and
$y:=x_2/y_2$. Suppose $\ord(x)\geq B$ and $\ord(y)\geq B$, or
$\ord(x)\leq -B$ and $\ord(y)\leq -B$ for some $B\geq 0$. Then the
distance from $\cA(\gamma)$ to $O$ is at least $B$.
\end{lem}
\begin{proof} Note that $x\neq y$, as $\begin{pmatrix} x\\ 1
\end{pmatrix}$ and $\begin{pmatrix} y\\ 1 \end{pmatrix}$ are
eigenvectors for $a$ and $b$. Without loss of generality we can
assume $\ord(x)\geq \ord(y)$. The geodesic $\cA(\gamma)$ has
vertices
$$
\begin{pmatrix} x\pi^k & y \\ \pi^k & 1\end{pmatrix}, \quad k\in \Z.
$$
First, consider the case when $k>0$. Then $\begin{pmatrix} 1 & 0 \\
-\pi^k & 1\end{pmatrix}\in \GL_2(\cO)$, so by multiplying
$\begin{pmatrix} x\pi^k & y \\ \pi^k & 1\end{pmatrix}$ by this
matrix from the right we see that
$$
\begin{pmatrix} x\pi^k & y \\ \pi^k &
1\end{pmatrix} \sim \begin{pmatrix} (x-y)\pi^k & y \\ 0 &
1\end{pmatrix}.
$$
On the other hand, $\ord((x-y)\pi^k)\geq k+\ord(y)>\ord(y)$, so the
resulting matrix is in reduced form. The distance from this matrix
to $O$ is
$$
\ord((x-y)\pi^k)\geq k+\ord(y), \quad \text{if }\ord(y)\geq 0,
$$
and
$$
\ord((x-y)\pi^k)-2\ord(y)\geq k-\ord(y), \quad \text{if }\ord(y)< 0.
$$
In either case, we conclude that the distance is at least
$1+|\ord(y)|\geq 1+B$.

The matrix with $k=0$ is adjacent to the matrix with $k=1$, so from
the previous paragraph we conclude that the corresponding matrix and
$O$ are at a distance at least $(1+B)-1=B$.

Now suppose $k<0$. Then
$$
\begin{pmatrix} x\pi^k & y \\ \pi^k &
1\end{pmatrix} \sim \begin{pmatrix} x & \pi^{-k}y \\ 1 &
\pi^{-k}\end{pmatrix}\sim \begin{pmatrix} \pi^{-k}(y-x) & x \\ 0 &
1\end{pmatrix}.
$$

If $\ord(x)=\ord(y)$, then $\ord(x-y)\geq \ord(y)=\ord(x)$. Hence
$-k+\ord(x-y)>\ord(x)$, and the above matrix is in reduced form. The
distance from $O$ is
$$
-k+\ord(x-y)\geq -k+\ord(y)>B, \quad \text{if }\ord(y)\geq 0,
$$
or
$$
-k+\ord(x-y)-2\ord(y)\geq -k-\ord(y)>B, \quad \text{if }\ord(y)< 0.
$$

If $\ord(x)>\ord(y)$, then $\ord(x-y)= \ord(y)$. If
$\ord(\pi^{-k}(x-y))=-k+\ord(y)\leq \ord(x)$, then
$$
\begin{pmatrix} \pi^{-k}(y-x) & x \\ 0 &
1\end{pmatrix}\sim \begin{pmatrix} \pi^{-k}(y-x) & 0 \\ 0 &
1\end{pmatrix},
$$
so the distance is $|-k+\ord(y)|$. If $\ord(y)\geq B$, then this
last quantity is obviously $\geq 1+B$. On the other hand, if
$\ord(y)\leq -B$, then, due to the assumption of the lemma,
$-k+\ord(y)\leq \ord(x)\leq -B$. Thus, $|-k+\ord(y)|\geq B$ Finally,
if $-k+\ord(y)> \ord(x)$, then the distance is
$$
-k+\ord(x-y)= -k+\ord(y)>B, \quad \text{if }\ord(x)\geq 0,
$$
or
$$
-k+\ord(x-y)-2\ord(x)= -k+\ord(y)-2\ord(x)>-\ord(x)\geq B,
$$
if $\ord(x)< 0$.
\end{proof}

\begin{rem}
The inequalities $\ord(x)\gg 0$, $\ord(y)\gg 0$ (resp. $\ord(x)\ll
0$, $\ord(y)\ll 0$) essentially mean that both $x$ and $y$ are in a
small neighborhood of $0$ (resp. $\infty$). Now one can visualize
the previous lemma as follows: for a sufficiently small interval on
$\R$ the geodesic in the hyperbolic upper half-plane $\cH$
connecting any two distinct points in that interval is far from a
fixed point in $\cH$.
\end{rem}

\begin{notn} For $B\geq 0$, let $\cT_B$ be the finite subtree of $\cT$ with
set of vertices
$$
\Ver(\cT_B)=\{v\in \Ver(\cT)\ |\ d(v,O)\leq B\}.
$$
\end{notn}

\begin{lem}\label{lem-empty} Let $n$ be the amplitude of translation
with which $\gamma$ acts on $\cA(\gamma)$. Let $m$ be the distance
from $\cA(\gamma)$ to $O$. If $2m+n> 2B$, then $\gamma\cT_B\cap
\cT_B=\emptyset$.
\end{lem}
\begin{proof} Let $P\in \cA(\gamma)$ be the vertex closest to $O$.
Then $m=d(P,O)$ and
$$
d(O, \gamma O)=d(O,P)+d(P, \gamma P)+d(\gamma P, \gamma O)=2m+n,
$$
cf. \cite[Prop. 24 (iv), p. 63]{SerreT}. On the other hand, if
$\gamma\cT_B\cap \cT_B\neq \emptyset$, then $d(O, \gamma O)\leq 2B$.
Thus $2m+n\leq 2B$.
\end{proof}

\begin{prop}\label{lemDiam}
Assume $\G$ acts without inversion on $\cT$ and $\G\bs \cT$ is
finite. Let $B=D(\G\bs \cT)$. Let $S$ denote the set of $\gamma\in
\G$ such that $\gamma\cT_B\cap \cT_B\neq \emptyset$. Then $S$
generates $\G$.
\end{prop}
\begin{proof} Let $\bar{O}\in \G\bs \cT$ be the image of $O$. Let
$\bar{v}\in \Ver(\G\bs \cT)$. Consider a path $P$ of shortest length
connecting $\bar{O}$ and $\bar{v}$. Obviously $P$ is a subtree of
$\G\bs \cT$, hence by \cite[Prop. 14, p. 25]{SerreT} lifts to $\cT$.
This implies that there is a vertex $v$ in $\cT$ which maps to
$\bar{v}$ and $d(O,v)=d(\bar{O},\bar{v})$. In particular, $v\in
\cT_B$, so $\cT_B$ surjects onto $\G\bs \cT$ under the quotient map
$\cT\to \G\bs \cT$.

Let $\mathrm{real}(\cT)$ be the \textit{realization} of $\cT$; see
\cite[p. 14]{SerreT}. Recall that $\mathrm{real}(\cT)$ is a
CW-complex where each edge of $\cT$ is homeomorphic to the interval
$[0,1]\subset \R$. Let $U$ be the open subset of
$\mathrm{real}(\cT)$ consisting of points at distance $< 1/3$ from
$\mathrm{real}(\cT_B)$. Then $\gamma U\cap U\neq \emptyset$ if and
only if $\gamma\cT_B\cap \cT_B$, and $U\to \mathrm{real}(\G\bs \cT)$
is surjective. The claim of the proposition now follows from
\cite[(1), p. 30]{SerreT}.
\end{proof}

Let $\G$ be as in Proposition \ref{prop2.5}. $\G$ acts naturally on
$\cT$ (see Section \ref{SecHE}). The action is without inversion and
the quotient graph $\G\bs\cT$ is finite; see \cite[Lem.
5.1]{PapLocProp}.

\begin{lem}\label{lem3.7}
Let $V:=\#\Ver(\G\bs \cT)$.
\begin{enumerate}
\item If $\Odd(R)=1$, then
$$
V= \frac{2}{(q-1)(q^2-1)}\prod_{x\in
R}(q^{\deg(x)}-1)+\frac{q}{q+1}2^{\#R-1}.
$$
\item If $\Odd(R)=0$, then
$$
V= \frac{2(q^{\deg(\fa)}+1)}{(q-1)(q^2-1)}\prod_{x\in
R}(q^{\deg(x)}-1).
$$
\end{enumerate}
\end{lem}
\begin{proof}
(1) This follows from \cite[Thm. 5.5]{PapLocProp}. (2) Since
$\F_q^\times\lhd \G$ acts trivially on $\cT$ and $\G/\F_q^\times$ is
a free group, $\G\bs \cT$ is $(q+1)$-regular. Thus, if we denote by
$E$ the number of (non-oriented) edges of $\G\bs \cT$, then
$E=(q+1)V/2$. On the other hand, by \cite[Thm. $4'$, p. 27]{SerreT},
the rank of $\G/\F_q^\times$ is equal to $E+1-V$. Now the expression
for $V$ follows from Proposition \ref{prop2.5}.
\end{proof}

\begin{lem}\label{LPS} Let $\cG$ be an $m$-regular Ramanujan graph on $n$
vertices. Then $$D(\cG)\leq 2\log_{m-1}(n)+\log_{m-1}(4).$$
\end{lem}
\begin{proof}
For the definition of Ramanujan graphs see \cite{Lubotzky}. The
claim of the lemma is part of Proposition 7.3.11 in \textit{loc.
cit.}
\end{proof}

\begin{prop}\label{prop3.9}
$D(\G\bs\cT)\leq 2\deg(\fa\fr)+3$.
\end{prop}
\begin{proof}
Let $I=(T)\lhd A$ be the ideal generated by $T$. Let $\Upsilon$ be a
maximal order containing $\La$. Denote by $\G(T)$ the principal
level-$I$ congruence subgroup of $\Upsilon^\times$, cf. \cite{LSV}.
Let $\G':=\G\cap \G(I)$. $\G\bs\cT$ is naturally a quotient of
$\G'\bs\cT$, so $$D(\G\bs\cT)\leq D(\G'\bs\cT).$$ Moreover, since
$[\G:\G']<q^4$, we have $V'< q^4 V$, where $V':=\#\Ver(\G'\bs\cT)$.

By \cite[Thm. 1.2]{LSV}, $\G'\bs\cT$ is a $(q+1)$-regular Ramanujan
graph, so the previous paragraph and Lemma \ref{LPS} give the bound
$$
D(\G\bs\cT)\leq 2\log_q(q^4 V)+1.
$$
Now the proposition follows from the formulae in Lemma \ref{lem3.7}.
(Note that $\deg(\fr)=\sum_{x\in R}\deg(x)$, and $\deg(\fa)=0$ when
$\Odd(R)=1$.)
\end{proof}


\section{Proof of Theorem \ref{thm1}}\label{SecHE}

Let $(\fa, \fb)=(\fa', \fr)$ (resp. $(\fa, \fb)=(\fr, \xi)$) if
$\Odd(R)=0$ (resp. $\Odd(R)=1$), with $\fa',\xi$ chosen as in Lemma
\ref{lemHab}. Let $\La$ be the standard order in $H(\fa, \fb)$ and
$\G:=\La^\times$. By Proposition \ref{prop2.5}, this is a finitely
generated group, and we would like to find an explicit set of
generators. We start by embedding $H(\fa, \fb)$ into $\M_2(K)$. The
map
$$
i\mapsto \begin{pmatrix} \sqrt{\fa} & 0 \\ 0 &
-\sqrt{\fa}\end{pmatrix}, \quad j\mapsto \begin{pmatrix} 0 & 1\\ \fb
& 0\end{pmatrix}
$$
gives an embedding $H(\fa, \fb)\hookrightarrow \M_2(K)$. Indeed,
since $\fa\in A$ is monic and has even degree, the equation
$x^2=\fa$ has a solution in $K$. Thus, the above matrices are indeed
in $\M_2(K)$. It remains to observe that the given matrices satisfy
the same relations as $i$ and $j$. Under this embedding $\G$ is the
subgroup of $\GL_2(K)$ consisting of matrices
$$
\G=\left\{\begin{pmatrix} a + b\sqrt{\fa} & c + d\sqrt{\fa}
\\ \fb(c - d\sqrt{\fa}) & a - b\sqrt{\fa} \end{pmatrix}\ \bigg|\
\begin{matrix} a,b,c,d,\in A \\ \det=a^2-b^2\fa-c^2\fb+d^2\fa\fb
\in \F_q^\times\end{matrix}\right\}.
$$

\begin{prop}\label{prop5.3} Fix some $B\geq 0$.
If $\gamma\in \G$ is hyperbolic and satisfies $\norm{\gamma}>2B$,
then $\gamma \cT_B\cap \cT_B=\emptyset$.
\end{prop}
\begin{proof} To simplify the notation, we put $\deg(0)=0$.
The image of $\gamma$ in $\GL_2(K)$ is the matrix
$$
\begin{pmatrix} a + b\sqrt{\fa} & c + d\sqrt{\fa}
\\ \fb(c - d\sqrt{\fa}) & a - b\sqrt{\fa} \end{pmatrix}.
$$
Let $\alpha, \beta\in K$ be the eigenvalues of $\gamma$. We know
that $\Nr(\gamma)=\alpha\beta=:\kappa\in \F_q^\times$ and
$\Tr(\gamma)=\alpha+\beta=2a\not\in \F_q$. Thus,
$\beta=\kappa\alpha^{-1}$ and $\ord(\beta)=-\ord(\alpha)$. Without
loss of generality, we assume $\ord(\alpha)\geq 0$, so
$\ord(\beta)\leq 0$. Since $\deg(a)\geq 1$, $\ord(\beta)=\ord(a)\leq
-1$. Using Lemma \ref{lem-geod}, we conclude that $\gamma$ acts on
$\cA(\gamma)$ by translations with amplitude $2\deg(a)$. Let $m$ be
the distance from $\cA(\gamma)$ to $O$. Lemma \ref{lem-empty}
implies that if
\begin{equation}\label{eq-empty}
m+\deg(a)>B
\end{equation}
then $\gamma \cT_B\cap \cT_B=\emptyset$. In particular, if
$\deg(a)>B$, then $\gamma \cT_B\cap \cT_B=\emptyset$.

Suppose $\begin{pmatrix} 1\\
0\end{pmatrix}$ is an eigenvector.  Then $\fb(c - d\sqrt{\fa})=0$,
which forces $c=d=0$. Thus, we must have $a^2-b^2\fa\in
\F_q^\times$. If $\deg(b)>\deg(a)-\deg(\fa)/2$, then this is not
possible. It is easy to see that we
reach the same conclusion in the case when $\begin{pmatrix} 0\\
1\end{pmatrix}$ is an eigenvector.

Assume $\deg(b)> \deg(a)-\deg(\fa)/2$. Denote
$s:=\deg(b)+\deg(\fa)/2-\deg(a)>0$. From the previous paragraph we
know that there are eigenvectors for
$\alpha$ and $\beta$ of the form $\begin{pmatrix} x\\
1\end{pmatrix}$ and $\begin{pmatrix} y\\
1\end{pmatrix}$, with $x\neq 0$, $y\neq 0$. Now
\begin{align*}
& x(a + b\sqrt{\fa}) +(c + d\sqrt{\fa})=\alpha x \\
& x\fb(c - d\sqrt{\fa}) +(a - b\sqrt{\fa})=\alpha.
\end{align*}
From this we get
\begin{equation}\label{eq4.1}
x=-\frac{c + d\sqrt{\fa}}{a +
b\sqrt{\fa}-\alpha}=\frac{\alpha-a+b\sqrt{\fa}}{\fb(c -
d\sqrt{\fa})}.
\end{equation}
Similarly,
\begin{equation}\label{eq4.2}
y=-\frac{c + d\sqrt{\fa}}{a +
b\sqrt{\fa}-\kappa\alpha^{-1}}=\frac{\kappa\alpha^{-1}-a+b\sqrt{\fa}}{\fb(c
- d\sqrt{\fa})}.
\end{equation}
Hence
\begin{equation}\label{eq1Prop}
\frac{x}{y}=\frac{a + b\sqrt{\fa}-\kappa\alpha^{-1}}{a +
b\sqrt{\fa}-\alpha}
\end{equation}
and
\begin{equation}\label{eq2Prop}
x-y=\frac{(c + d\sqrt{\fa})(\kappa\alpha^{-1}-\alpha)}{(a +
b\sqrt{\fa}-\alpha)(a +
b\sqrt{\fa}-\kappa\alpha^{-1})}=x\frac{\alpha-\kappa\alpha^{-1}}{(a
+ b\sqrt{\fa}-\kappa\alpha^{-1})}.
\end{equation}
From our assumption $\deg(b)> \deg(a)-\deg(\fa)/2$ and
(\ref{eq1Prop}), we get
$$
\ord(x)=\ord(y)=:n.
$$
Similarly from (\ref{eq2Prop}), we get
$$
\ord(x-y)=n+s.
$$
By (\ref{eq-geod}), the geodesic $\cA(\gamma)$ has vertices
$$
\Ver(\cA(\gamma))=\left\{\begin{pmatrix} x\pi^k & y \\ \pi^k &
1\end{pmatrix}\ |\ k\in \Z\right\}.
$$
If $k\geq 0$, then
$$
\begin{pmatrix} x\pi^k & y \\ \pi^k &
1\end{pmatrix}\sim \begin{pmatrix} (x-y)\pi^k & y \\ 0 &
1\end{pmatrix}\sim \begin{pmatrix} \pi^{k+n+s} & y \\ 0 &
1\end{pmatrix}.
$$
Since $y\neq 0$ and $k+n+s>n=\ord(y)$, the last matrix is in reduced
form. By Lemma \ref{lem2.1}, the distance of this matrix from $O$ is
$k+|n|+s\geq s$. If $k<0$, then
$$
\begin{pmatrix} x\pi^k & y \\ \pi^k &
1\end{pmatrix}\sim \begin{pmatrix} x & y\pi^{-k} \\ 1 &
\pi^{-k}\end{pmatrix}\sim \begin{pmatrix} y\pi^{-k} & x \\
\pi^{-k} & 1\end{pmatrix}.
$$
A similar calculation shows that the distance from this matrix to
$O$ is again greater or equal to $s$. Now
$$
s+\deg(a)=\deg(b)+\deg(\fa)/2.
$$
If $\deg(b)>B-\deg(\fa)/2$, then this last quantity is greater than
$B$. Therefore, by (\ref{eq-empty}), $\gamma \cT_B\cap
\cT_B=\emptyset$. Note that if $\deg(a)\leq B$, then
$\deg(b)>B-\deg(\fa)/2$ implies $\deg(b)>\deg(a)-\deg(\fa)/2$.
Overall, we have shown that if $\deg(a)> B$ or
$\deg(b)>B-\deg(\fa)/2$, then $\gamma \cT_B\cap \cT_B=\emptyset$.

Now assume $\deg(a)\leq B$ and $\deg(b)\leq B-\deg(\fa)/2$, but
$\deg(c)> 2B$ or $\deg(d)> 2B$. From our earlier
discussion we know that $\begin{pmatrix} 1\\
0\end{pmatrix}$ or $\begin{pmatrix} 0\\
1\end{pmatrix}$ are not eigenvectors (since otherwise $c=d=0$). Now,
either $\ord(c+d\sqrt{\fa})<-2B$ or $\ord(c-d\sqrt{\fa})<-2B$ (and
only one of the inequalities holds due to $\Nr(\gamma)\in
\F_q^\times$). First, assume $\ord(c+d\sqrt{\fa})<-2B$. Since
$\ord(a + b\sqrt{\fa}-\alpha)\geq -B$ and $\ord(a +
b\sqrt{\fa}-\kappa\alpha^{-1})\geq -B$, from the first equality in
(\ref{eq4.1}) and (\ref{eq4.2}) we get $\ord(x)< -B$ and
$\ord(y)<-B$. Next, assume $\ord(c-d\sqrt{\fa})<-2B$. Since
$\ord(\alpha-a+b\sqrt{\fa})\geq -B$ and
$\ord(\kappa\alpha^{-1}-a+b\sqrt{\fa})\geq -B$, from the second
equality in (\ref{eq4.1}) and (\ref{eq4.2}) we get $\ord(x)> B$ and
$\ord(y)>B$. In either case, Lemma \ref{lem-BB} implies that the
distance from $\cA(\gamma)$ to $O$ is greater than $B$. As before,
from (\ref{eq-empty}) we conclude $\gamma \cT_B\cap
\cT_B=\emptyset$.
\end{proof}

\begin{prop}\label{prop4.5}
If $\gamma\in \G$ is elliptic and satisfies $\norm{\gamma}>B$, then
$\gamma \cT_B\cap \cT_B=\emptyset$.
\end{prop}
\begin{proof} To simplify the notation, we again put $\deg(0)=0$.
If $\gamma$ is elliptic and satisfies the inequality
of the proposition, then obviously $\gamma\not \in \F_q^\times$. By
Proposition \ref{prop2.5}, the existence of such $\gamma$ is
possible if and only if $\Odd(R)=1$. Hence, we can assume that
$$
\gamma=\begin{pmatrix} a + b\sqrt{\fr} & c + d\sqrt{\fr}
\\ \xi(c - d\sqrt{\fr}) & a - b\sqrt{\fr} \end{pmatrix}.
$$
Since $\gamma$ is elliptic, $a\in \F_q$. Assume $\deg(c)>B$.
Consider $\tau:=\gamma\cdot j\in \G$:
$$
\tau=\begin{pmatrix}  \xi(c + d\sqrt{\fr}) & a +
b\sqrt{\fr} \\
\xi(a - b\sqrt{\fr}) & \xi(c - d\sqrt{\fr}) \end{pmatrix}.
$$
Note that $\tau$ is hyperbolic since $\Tr(\tau)=2\xi c\not \in
\F_q$. Suppose $\gamma \cT_B\cap \cT_B\neq \emptyset$. Since $j$
fixes $O$, we get
$$
d(O,\tau O)=d(O, \gamma O)\leq 2B.
$$
On the other hand, $\tau$ acts on its geodesic by translation with
amplitude $2\deg(c)>2B$, so $d(O,\tau O)>2B$ - a contradiction.
Thus, from now on we can assume $\deg(c)\leq B$. If $\deg(b)> B$,
then $\deg(a^2-b^2\fr)>2B$. Since
$$
(a^2-b^2\fr)-\xi(c^2- d^2\fr) \in \F_q^\times,
$$
we must have $\deg(b)=\deg(d)=:s$. In this case we have
$$
\gamma=\begin{pmatrix} z_1\pi^{-s} & z_2 \pi^{-s}\\
z_3 \pi^{-s} & z_4\pi^{-s}
\end{pmatrix}\sim \begin{pmatrix} z_1 & z_2 \\
z_3  & z_4
\end{pmatrix},
$$
where $z_1,\dots, z_4\in \cO^\times$. Moreover, since
$\det(\gamma)\in \F_q^\times$,$$z_1z_4-z_2z_3\in
\pi^{2s}\cO^\times.$$ This implies
$$
\begin{pmatrix} z_1 & z_2 \\ z_3  & z_4 \end{pmatrix}
\sim \begin{pmatrix} \pi^{2s} & 0 \\ 0  & 1 \end{pmatrix}.
$$
The distance from this matrix to $O$ is obviously $2s$. On the other
hand, the matrix corresponding to the vertex $\gamma O$ under
(\ref{eq-setM}) is $\gamma$, so $d(O, \gamma O)=2s>2B$. The previous
argument works also under the assumption $\deg(d)> B$, and leads to
the same conclusion.
\end{proof}

\begin{proof}[Proof of Theorem \ref{thm1}] Let $B=2\deg(\fa\fb)+3$. By Proposition \ref{prop3.9},
$D(\G\bs \cT)\leq B$. By Proposition \ref{prop5.3} and
\ref{prop4.5}, the set of elements $\gamma\in \G$ such that $\gamma
\cT_B\cap \cT_B\neq \emptyset$ is contained in $S$. Now the theorem
follows from Proposition \ref{lemDiam}.
\end{proof}


\subsection*{Acknowledgments} The work on this paper was started while I was visiting the
Department of Mathematics of Saarland University. I thank the
members of the department for their warm hospitality. I thank R.
Butenuth and E.-U. Gekeler for stimulating discussions.



\begin{thebibliography}{10}

\bibitem{AB}
M.~Alsina and P.~Bayer, \emph{Quaternion orders, quadratic forms and
{Shimura}
  curves}, Amer. Math. Soc., 2004.

\bibitem{CK}
J.~Chalk and B.~Kelly, \emph{Generating sets for {Fuchsian} groups},
Proc. Roy.
  Soc. Edinburgh Sect. A \textbf{72} (1975), 317--326.

\bibitem{Improper}
E.-U. Gekeler, \emph{Improper {Eisenstein} series on {Bruhat-Tits}
trees},
  manuscripta math. \textbf{86} (1995), 367--391.

\bibitem{Johansson}
S.~Johansson, \emph{On fundamental domains of arithemtic {Fuchsian}
groups},
  Math. Comp. \textbf{69} (2000), 339--349.

\bibitem{Lubotzky}
A.~Lubotzky, \emph{Discrete groups, expanding graphs and invariant
measures},
  Birkh\"auser, 1994.

\bibitem{LSV}
A.~Lubotzky, B.~Samuels, and U.~Vishne, \emph{Ramanujan complexes of
type
  $\tilde{A}_d$}, Israel J. Math. \textbf{149} (2005), 267--299.

\bibitem{Miyake}
T.~Miyake, \emph{Modular forms}, Springer-Verlag, 1989.

\bibitem{PapLocProp}
M.~Papikian, \emph{Local diophantine properties of modular curves of
  {$\mathcal{D}$}-elliptic sheaves}, submitted for publication, available at
  http://www.math.psu.edu/papikian.

\bibitem{Rosen}
M.~Rosen, \emph{Number theory in function fields}, Springer, 2002.

\bibitem{SerreLF}
J.-P. Serre, \emph{Corps locaux}, Hermann, 1968.

\bibitem{SerreT}
J.-P. Serre, \emph{Trees}, Springer Monographs in Math., 2003.

\bibitem{Vigneras}
M.-F. Vign\'eras, \emph{Arithm\'etiques des alg\`ebres de
quaternions}, LNM
  800, 1980.

\end{thebibliography}

\end{document}